\documentclass{amsart}

\usepackage{amsmath, amssymb, graphicx, epsfig, verbatim}
\usepackage{epic}

\newtheorem{thm}{Theorem}[section]

\newtheorem{cor}[thm]{Corollary}
\newtheorem{lem}[thm]{Lemma}
\newtheorem{prop}[thm]{Proposition}

\newtheorem{defn}[thm]{Definition}
\newtheorem{conj}[thm]{Conjecture}

\newtheorem{rem}[thm]{Remark}

\numberwithin{equation}{section}
\newcommand{\defin}[1]{{\bf\emph{#1}}}
\newcommand{\Field}{{\mathbb {F}}}
\newcommand{\HFKa}{{\widehat {\rm {HFK}}}}

\newcommand{\posit}{\mathcal P}
\newcommand{\negat}{\mathcal N}

\newcommand{\Z}{\mathbb Z}

\newcommand{\Q}{\mathbb Q}

\begin{document}

\title{Purely cosmetic surgeries and pretzel knots}

\author{Andr\'{a}s I. Stipsicz}
\address{R\'enyi Institute of Mathematics\\
H-1053 Budapest\\ 
Re\'altanoda utca 13--15, Hungary}
\email{stipsicz.andras@renyi.hu}

\author{Zolt\'an Szab\'o}
\address{Department of Mathematics\\
Princeton University,\\
 Princeton, NJ, 08544}
\email{szabo@math.princeton.edu}

\begin{abstract}
 We show that all pretzel knots satisfy the (purely)
 cosmetic surgery conjecture, i.e.  Dehn surgeries with different slopes
 along a pretzel knot provide different oriented three-manifolds.
\end{abstract}
\maketitle

%\end{document}

\section{Introduction}
\label{sec:intro}
Suppose that $K\subset S^3$ is a knot in the three-sphere and $r\in
{\mathbb {Q}}$ a rational number.  Let $S^3_r (K)$ denote the effect
of Dehn surgery along $K$ with coefficient $r$.  The Purely Cosmetic
Surgery Conjecture (PCSC for short) asserts:
\begin{conj}[PCSC]
  For every nontrivial knot $K$, the orientation-preserving diffeomorphism
  $S_s^3(K)\cong S_r^3(K)$ for $s,r\in \Q$ implies that $s=r$.
  \end{conj}

The conjecture has been verified for 2-bridge knots \cite{twobridge},
for connected sums~\cite{connectedsum}, for 3-braid
knots~\cite{Varva2}, for knots of Seifert genus one~\cite{Wang}
and for prime knots with at most 16
crossings~\cite{Hanselman}. By the classification of Seifert fibered
spaces, the conjecture also holds for torus knots. Note that $K$ and
its mirror image $m(K)$ satisfies the conjecture at the same time,
since $S^3_r(m(K))=-S^3_{-r}(K)$.

When we relax the condition that the diffeomorphism is
orientation-preserving, there are some examples of knots admitting
diffeomorphic surgeries with different slopes: for example, for an
amphichiral knot $K$ we have that $S^3_r(K)$ and $S^3 _{-r}(K)$ are
diffeomorphic. See \cite{Varva} for further results, including  
theorems for preztel knots.

Suppose that $P=P(a_1, \ldots , a_n)$ is a pretzel knot with $n$ strands,
where $a_i$ denotes the number of half-twists (right-handed for
positive and left-handed for negative $a_i$) on the $i^{th}$ strand,
see Figure~\ref{fig:Pretzel} for an illustration.

Our main result is the verification of PCSC for pretzel
knots:
\begin{thm}\label{thm:main1}
  The Purely Cosmetic Surgery Conjecture holds
  for pretzel knots.
 \end{thm}

\begin{figure}
\centering
\includegraphics{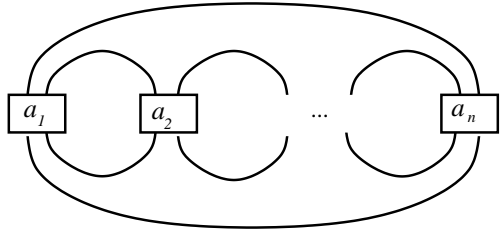}
\caption{{\bf The pretzel knot $P(a_1, \ldots , a_n)$.}  In the
  following we will assume that $a_2, \ldots , a_n$ are odd, and $a_1$
  is either even or odd. In order to have a knot, if $a_1$ is odd,
  then $n$ must also be odd.}
\label{fig:Pretzel} 
\end{figure}
In the following we will always assume that $P$ is a knot, implying that
either 
\begin{itemize}
\item all $a_i$ are odd and $n$ is odd, or
\item exactly one $a_i$ (which can be assumed to be $a_1$) is even, and $n$ is odd, or
\item exactly one $a_i$ (which can be assumed to be $a_1$) is even, and $n$ is even.
\end{itemize}

Note that the order of the $a_i$'s in defining the pretzel knot
$P=P(a_1, \ldots , a_n)$ is important, and in general can be changed
only by the action of the dihedral group (when $P$ is viewed in the
isotopic position shown by Figure~\ref{fig:pretzel2}).
\begin{figure}
\centering
\includegraphics{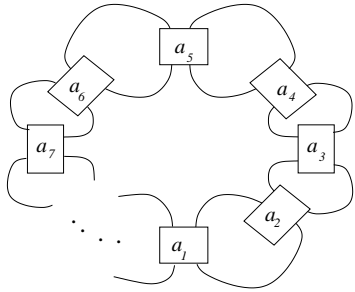}
\caption{{\bf The dihedral action is more visible in this diagram of
    $P(a_1, \ldots , a_n$.}
  The boxes are positioned at the vertices of a regular $n$-gon.}
\label{fig:pretzel2} 
\end{figure}
One noteable exception is that if $a_i=\pm 1$ then it can be commuted
with any other strand (by rotating the two strands together), hence
these can be collected at the end of the string. In addition, there
are two cases when the number of strands can be reduced: if $a_i=1$
and $a_{i+1}=-1$ then these two strands can be eliminated by a simple
isotopy (a Reidemeister 2 move); and if $a_1=2$ and $a_2=-1$ (or if
$a_1=-2$ and $a_2=1$) then the isotopy shown by
Figure~\ref{fig:simplify} reduces the number of strands by one.
\begin{figure}
\centering
\includegraphics{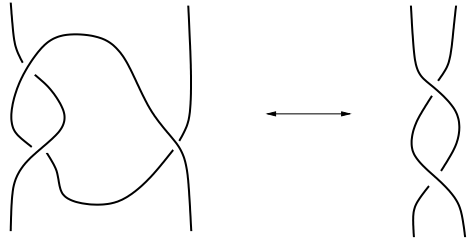}
\caption{ {\bf The isotopy above shows that $(2,-1)$ in any string
    $(a_1, \ldots , a_n)$ defining the pretzel knot $P(a_1, \ldots ,
    a_n)$ can be replaced by $(-2)$.}}
\label{fig:simplify} 
\end{figure}
For this reason, in the following we will always
assume that $\{ 1, -1\}, \{2 ,-1\}$ and $\{ -2,1\}$ are not subsets of
$\{ a_i\}_{i=1}^n$. Furthermore we will always assume that $a_i\neq
0$, since when $a_1=0$, the knot $P$ is the connected sum of
alternating torus knots, and for connected sums the conjecture has
already been verified~\cite{connectedsum}. In a similar manner, we
will always assume that $n\geq 3$, since two-strand pretzel knots are
(alternating) torus knots, and for those the conjecture is known to
hold true. 

The paper is organized as follows.  In Section~\ref{sec:Obstructions}
we collect some obstructions stemming from the Alexander and Jones
polynomials for knots to support purely cosmetic surgeries.  In
Section~\ref{sec:KnotFloer} we observe that pretzel knots have (knot
Floer homology) thickness at most one.  In Section~\ref{sec:genera}
some background regading Seifert genera of pretzel knots is given.
(In the light of a recent result of Hanselman~\cite{Hanselman} to be
discussed later, Seifert genera are of central importance in deriving
statements regarding cosmetic surgeries.)  In
Section~\ref{sec:generalStrands} we deal with $n$-strand pretzel knots
with $n\neq 5$, and in Section~\ref{sec:FiveStrand} we deal with
five-strand pretzel knots and complete the proof of
Theorem~\ref{thm:main1}. We include a short Appendix providing a
computational scheme for the Jones polynomial of some pretzel knots.

\bigskip

{\bf {Acknowledgements}}: AS was partially supported by the
\emph{\'Elvonal (Frontier) project} of the NKFIH (KKP126683).  ZSz was
partially supported by NSF Grants DMS-1606571 and DMS-1904628.  The
second author would like to thank Konstantinos Varvarezos for helpful
discussions.

\section{Obstructions for purely cosmetic surgeries}
\label{sec:Obstructions}

%The Alexander and the Jones polynomial both provide obstructions
%for a knot to violate  PCSC. In this section we collect some of those
%obstructions.

A general result of Ni-Wu~\cite[Theorem~1.2]{NiWu}
provides strong constraints on the surgery coefficients potentially
providing cosmetic surgeries.
\begin{thm}[Ni-Wu]
\label{thm:NiWuMain}
  Suppose that $K\subset S^3$ is a nontrivial knot and for $r,s\in \Q$
  we have that $S^3_r(K)$ and $S^3_s(K)$ are orientation preserving
  diffeomorphic. Then $s=-r$ and if $r=\frac{p}{q}$ with $p,q>0$ relatively
  prime integers, then $q^2\equiv -1 \pmod{p}$. \qed
  \end{thm}

The Casson-Walker invariants of the three-manifolds $S^3_r(K)$ and
$S^3_{-r}(K)$ can be shown to be different (hence distinguish these
oriented three-manifolds) provided the Alexander polynomial $\Delta
_K(t)$ of $K$ satisfies a certain condition. More precisely, $\Delta
_K(t)$ provides the following obstruction for $K$ to admit purely
cosmetic surgeries.
  
  \begin{thm}(\cite[Proposition~5.1]{BoyerLines})
\label{thm:BoyerLines}
    If $K\subset S^3$ admits purely cosmetic surgeries, then for the
    Alexander polynomial $\Delta _K (t)$ we have $\Delta ''_K(1)=0$.\qed
  \end{thm}

  Here $\Delta _K(t)$ is defined by the skein relation 
  \begin{equation}\label{eq:skein}
  \Delta _{L_+}(t)-\Delta _{L_-}(t)=(t^{\frac{1}{2}}-t^{-\frac{1}{2}})\Delta _{L_0}(t)
  \end{equation}
  with $(L_+, L_-, L_0)$ forming a usual skein triple, and $\Delta$
  being normalized to 1 on the unknot. (Then $\Delta _K$ satisfies
  that $\Delta _K(1)=1, \Delta _K'(1)=0$ and $\Delta _K
  (t^{-1})=\Delta _K (t)$.) Indeed, this obstruction can be
  conveniently reformulated in terms of the Conway polynomial $\nabla
  _K(z)$ of $K$, where $\nabla _K$ can be described by the identity
  \[
  \nabla _K(t^{\frac{1}{2}}-t^{-\frac{1}{2}})=\Delta _K(t).
  \]
  In fact, the Conway polynomial can also be defined by a skein
  relation:
  \[
  \nabla _{L_+}(z)-\nabla _{L_-}(z)=z\nabla _{L_0}(z)
  \]
  for the skein triple $(L_+,L_-, L_0)$, normalized as 1 on the unknot.
  For a knot $K$, we have that $\nabla _K(z)=1+ \sum _{i=1}^d a_{2i}(K)z^{2i}$,
  and it is easy to see that $2a_2(K)=\Delta _K ''(1)$.
  For a two-component (oriented) link $L=K_1\cup K_2$ we have that
  $\nabla _L(z)=\sum _{i=0}^da_{2i+1}(L)z^{2i+1}$, and $a_1(L)=\ell k (K_1, K_2)$,
  the linking number of the two components, cf. \cite[Proposition~8.7]{Lick}.
  
%  Another obstruction for admitting purely cosmetic surgeries stems
%  from the Jones polynomial of a knot $K$. This obstruction rests on
  The three-manifold invariant $\lambda _2$ discussed in
  \cite{Lescop}, together with the surgery formula of
  \cite[Theorem~7.1]{Lescop} for $\lambda _2(S^3_r(K))$ in terms of
  the knot invariant $w_3(K)$ also provides an obstruction for
  cosmetic surgeries, leading to the following result:

  \begin{thm}(\cite[Proposition~3.4]{IchiharaWu})
    \label{thm:w3application}
  Suppose that $K\subset S^3$ is a knot with $a_2(K)=0$ and $p,q$ are
  postive integers with $q^2\equiv -1 \pmod{p}$. Then $\lambda _2
  (S^3_{\frac{p}{q}}(K))=\lambda _2( S^3_{-\frac{p}{q}}(K))$ if and only if
  $w_3(K)=0$. \qed
  \end{thm}

  The invariant $w_3(K)$ satisfies the following crossing change formula:
  if $(K_+,K_-, K'\cup K'')$ is a skein triple involving two knots
  $K_{\pm}$ and the two-component link $K'\cup K''$, then
  \[
  w_3(K_+)-w_3(K_-)=\frac{1}{2}(a_2(K')+a_2(K''))-\frac{1}{4}(a_2(K_+)+
  a_2(K_-)+\ell k ^2 (K',K'')),
  \]
  where (as usual) $\ell k (K',K'')$ is the linking number of the two
  (oriented) knots $K',K''$.

  \begin{rem}
    Indeed, both knot invariants above can be conveniently presented in terms of
    the Jones polynomial $V_K(t)$ of the knot $K$. (Here we consider
    the Jones polynomial satisfying the skein relation
    $t^{-1}V_{L_+}(t)-tV_{L_-}(t)=(t^{\frac{1}{2}}-t^{-\frac{1}{2}})V_{L_0}(t),$
    normalized as 1 on the unknot.) Indeed, since $6a_2(K)=3\Delta ''_K(1)=
    -V''_K(1)$ and by \cite[Lemma~2.2]{IchiharaWu} 
\[
w_3(K)=\frac{1}{72}V_K'''(1)+\frac{1}{24}V''_K(1)
\]
 holds, the above obstructions can be summarized as was done in
 \cite[Theorem~1.1]{IchiharaWu}: if $K \subset S^3$ admits
 purely cosmetic surgeries then
 $V_K''(1)=0$ and $V_K'''(1)=0$.
  \end{rem}

\section{Knot Floer homology of pretzel knots}
\label{sec:KnotFloer}

Heegaard Floer homology can be used in more than one way to
verify that a knot satisfies PCSC. The concordance invariant
$\tau$ (introduced in \cite{OSzGenus}) provides the following
obstruction:

\begin{thm}(\cite[Theorem~1.2(c)]{NiWu})
  \label{thm:niwu}
  If the tau-invariant $\tau (K)$ of the knot $K\subset S^3$ derived
  from knot Floer homology is not equal to 0, then $K$ satisfies PCSC.\qed
  \end{thm}
%Indeed, there are further results relying on Heegaard Floer homology,
%which we will spell out in Section~\ref{sec:KnotFloer} (after
%introducing some notation). Indeed, based on those results it will be
%obvious that

The hat version of knot Floer homology (over the field $\Field$ of
two elements) of a knot $K\subset S^3$ is a finite dimensional
bigraded vector space $\HFKa (K)=\sum _{M,A} \HFKa _M (K, A)$. By
collapsing the two gradings to $\delta =A-M$, we get the
$\delta$-graded invariant ${\HFKa }^{\delta}(K)$.
\begin{defn}
  The \defin{thickness} $th(K)$ of the knot $K\subset S^3$
  is the maximal value of the difference
  $\vert \delta (x)-\delta (x')\vert$ for homogeneous elements
  $x,x'\in {\HFKa }^{\delta }(K)$. In particular, if $th (K)=0$ then
  $K$ is called \defin{thin}.
    \end{defn}

Examples of thin knots are provided by alternating knots, where the
difference $A-M$ of a homogeneous element is equal to half the
negative of the signature of the knot.

Work of Hanselman~\cite{Hanselman} regarding PCSC is crucial in
our discussions.  In particular, a direct consequence of
\cite[Theorem~2]{Hanselman} is
\begin{cor}[Hanselman~\cite{Hanselman}]
  \label{cor:HanselmanThickness}
  If a nontrivial knot $K\subset S^3$ has thickness
  $th(K)\leq 5$ and $g(K)\neq 2$, then  PCSC holds for $K$.
  \end{cor}
\begin{proof}
  By the result of Wang~\cite{Wang} (see Theorem~\ref{thm:wang}), together
  with~\cite[Theorem~2]{Hanselman} of Hanselman, the orientation-preserving
  diffeomorphism $S^3_s(K)\cong S^3_r(K)$ for a nontrivial knot $K$
  and $r\neq s$ implies that $g(K)>1$ and 
  \begin{itemize}
  \item   either $\{ r,s\}=\{\pm 2\}$ and $g(K)=2$, or
    \item $\{ r, s\}=\{ \pm \frac{1}{q}\}$ for some positive integer $q$
     which  satisfies $q\leq \frac{th(K)+2g(K)}{2g(K)(g(K)-1)}$.
  \end{itemize}
  For a knot with $g(K)\neq 2$ the first option is not possible, and
  if $th(K)\leq 5$ and $g(K)\geq 3$,
  we get that the positive
  integer $q$ satisfies $q\leq \frac{11}{12}$, concluding the proof.
  \end{proof}

%  As the next result shows, pretzel knots have thickness at most 1, hence
%  all pretzel knots with Seifert genus distinct from 2 satisfy 
%  PCSC.
  %From Corollary~\ref{cor:HanselmanThickness} it follows then:

  \begin{prop}
    \label{prop:thickness}
    Suppose that $P=P(a_1, \ldots , a_n)$ is an $n$-strand pretzel knot.
    Then the thickness $th(P)$ of $P$ is at most 1.
\end{prop}
\begin{proof}
  We will show that there is a $\delta$-graded chain complex computing
  ${\HFKa }^{\delta}$ for which the thickness is at most 1, hence the
  same applies to the homologies. This chain complex is generated by
  the Kauffman states of the usual diagram of the pretzel knot
  $P=P(a_1, \ldots , a_n)$; we only need to determine the
  $\delta$-gradings of these generators.  (For the definition and
  basic properties of Kauffman states, as well as that they span a
  chain complex computing knot Floer homology, see
  \cite{OSzAlternating}.)  There are three types of domains in the
  diagram of $P$ from which the contributions should be counted:
    %; in this description we will
  %refer to the diagram on the left of Figure~\ref{fig:Pretzel}.  the
  bigons in the strands, domains between the strands, and the 'top
  domain'. (Notice that the 'bottom domain' and the outside unbounded
  domain does not have to be considered, since these are occupied by
  the marking, which is placed on the lower arc of the diagram.)
  Since the orientation of the strands is important in these
  calculations, we distinguish three cases.  These combinatorially
  different cases (together with the markings, symbolized by a heavy
  dot) and the orientations are shown by Figure~\ref{fig:Iranyitas1}.
\begin{figure}
\centering
\includegraphics{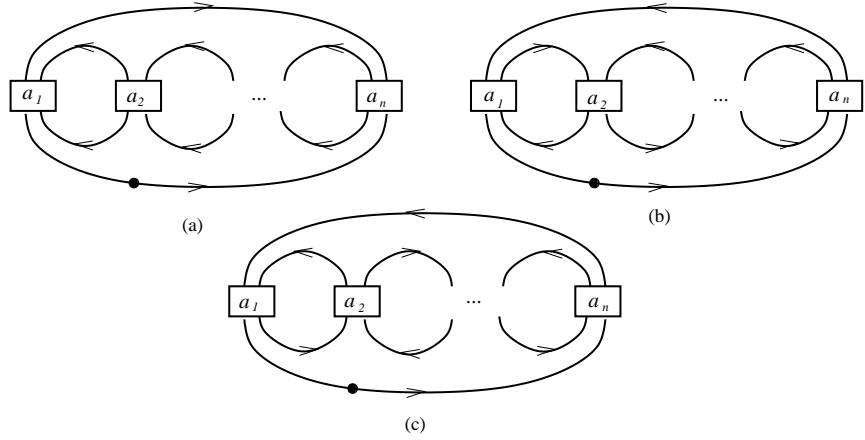}
\caption{{\bf Orientation on $P$.} The three diagrams indicate the three
  combinatorially different orientations: in (a) we show the case when
  all $a_i$ are odd (hence $n$ is odd), in (b) the case when $a_1$ is
  even and $n$ is odd, and finally in (c) the case when $a_1$ is even
  and $n$ is even. (The difference between the two last cases is the
orientation at the first strand.)}
\label{fig:Iranyitas1} 
\end{figure}

  Consider now a Kauffman state $\kappa$.
  The local contributions to $\delta$ are shown by
  Figure~\ref{fig:Contributions}; notice that the orientations of the
  strands are important in these calculations, hence the three cases
  shown by Figure~\ref{fig:Iranyitas1} should be discussed separately.
\begin{figure}
\centering
\includegraphics{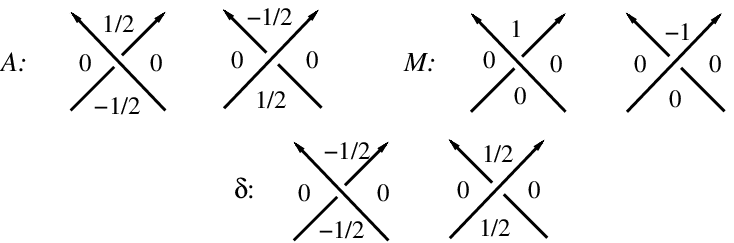}
\caption{{\bf The local contributions for $A, M$ and $\delta$ at a
    crossing.}  The Kauffman state distinguishes a corner at the
  crossing, and we take the value in that corner as a contribution of
  the crossing in $A,M$ or $\delta$ of the Kauffman state at hand.}
\label{fig:Contributions} 
\end{figure}

  \smallskip

  {\bf {Case I}}: \emph{All $a_i$ are odd.} In this case the
  orientation of $P$ can be chosen as shown by
  Figure~\ref{fig:Iranyitas1}(a). (Since $P$ is a knot, $n$ is odd.)
  The contribution of the marking of the Kauffman state $\kappa$ in
  the top domain, as well as in all bigons is 0. The domains between
  the strands, on the other hand, contribute either $\frac{1}{2}$ or
  $-\frac{1}{2}$, depending whether the marking is on the strand with
  positive or negative twisting. The fact whether the marking of such
  a domain is on the left or right strand is determined by the strand
  distinguished by the marking in the top domain. Therefore the sign
  of this distinguished strand determines how many $\frac{1}{2}$ or
  $-\frac{1}{2}$ contributions do we get.  Consequently, if there are
  $k$ negative and $\ell$ positive coefficients among the parameters
  $a_i$ of the pretzel knot $P$, the $\delta$-grading of $\kappa$ is
  either $\frac{1}{2}(k-\ell -1)$ (if the marking of the top domain is
  at a strand with negative parameter) or $\frac{1}{2}(k-\ell +1)$ (if
  the marking in the top domain is at a strand with positive
  parameter).  In conclusion there are at most two $\delta$-gradings,
  which are one apart, hence the thickness of the knot is at most 1.
  Indeed, if all $a_i$ have the same sign, then the knot is thin, in
  accordance with the fact that in that case the knot is alternating.

{\bf {Case II}}: \emph{Assume now that $a_1$ is even and $n$ is odd},
shown by the diagram of Figure~\ref{fig:Iranyitas1}(b).  In this case
the first strand (with the even parameter $a_1$) is special.  Bigons
in the first strand contribute 0, while in the other strands bigons
contribute $\pm \frac{1}{2}$ (the sign depending on the sign of the
parameter of the strand).  Consequently the bigons contribute to the
$\delta$-grading of $\kappa$ a fix value independent of the Kauffman
state, determined by the diagram only.  The top domain provides 0 if
the marking is at the first strand, and all the other domains give
further 0's. If the marking in the top domain is not at the first
strand, then its contribution is $\pm \frac{1}{2}$ (the sign depending
on the sign of the parameter), while now the domain between the first
and the second strand will have a nonzero contribution (which is again
$\pm \frac{1}{2}$, depending on the sign of $a_1$); call this
contribution $c$. Then the total contributions from the top domain and
the ones between the strands is either 0, or $-\frac{1}{2}+c$ or
$\frac{1}{2}+c$.  Since $c=\pm \frac{1}{2}$, the $\delta$-grading
still takes two possible values which are one apart, implying that
$th\leq 1$.

  {\bf {Case III}}: \emph{Finally, assume that $a_1$ is even and $n$
    is even}, cf. the diagram of Figure~\ref{fig:Iranyitas1}(c).  The
  only difference between this and the previous case is that the
  orientation along the first strand (with $a_1$ twists) is
  different. This case is similar to {\bf {Case I}}: all bigons contribute
  $\pm \frac{1}{2}$ (sign depending on the sign of the parameter of
  the strand), the top domain contributes $\pm \frac{1}{2}$ (depending
  on the fact whether the marking is on the top of a positive or a
  negative strand), while the contribution of the domains between the
  strands is all 0.  Once again, there are two possible
  $\delta$-values, which are 1 apart, verifying the claim.
\end{proof}
As a direct consequence of Corollary~\ref{cor:HanselmanThickness} we
have
\begin{cor}
\label{cor:Hcor}
  Suppose that $P=P(a_1, \ldots , a_n) $ is an $n$-strand pretzel knot.
If the Seifert genus $g(P)\neq 2$ then the purely cosmetic surgery
conjecture holds for $P$. \qed
\end{cor}

\section{Genera of pretzel knots}
\label{sec:genera}
The Seifert genera of knots play an important role in understanding
cosmetic surgeries on them.  Regarding low genus knots, the following
general result of Wang provides relevant information.
\begin{thm}(\cite[Theorem~1.3]{Wang})
\label{thm:wang}
If $g(K)=1$ for a knot $K$ then PCSC holds for $K$. \qed
\end{thm}
  
For Seifert genera of pretzel knots, we quote three results, detailed
below.  As before, we will assume that for the pretzel knot $P(a_1,
\ldots , a_n)$ we have that $\{1,-1\}, \{ -2, 1\}$ and $\{ 2, -1\}$
are not subsets of $\{ a_i\}_{i=1}^n$.

\subsection{Three-strand pretzel knots}

\begin{thm}(Kim-Lee, \cite[Corollary~2.7]{genera-pretzel})
  \label{thm:3stcase}
  The Seifert genus $g(P(p,q,r))$ of the three-strand pretzel knot
  $P(p,q,r)$ with parameters $p,q,r\in \Z\setminus \{ 0\}$ (also
  satisfying that $\{ 1,-1\},\{ 2, -1\}$ and $\{ -2, 1\}$ are not
    subsets of $\{ p,q,r\}$) is equal to
  \begin{enumerate}
  \item 1 if all $p,q,r$ are odd,
  \item $\frac{1}{2}(\vert q\vert +\vert r\vert )$ if $p$ is even and
    $q,r$ have the same sign, and 
  \item $\frac{1}{2}(\vert q\vert +\vert r\vert -2)$ if $p$ is even and
    $q, r$ have opposite signs.\qed
  \end{enumerate}
\end{thm}
A three-strand pretzel knot $P=P(p,q,r)$ with all odd coefficients
therefore satisfies PCSC by Theorem~\ref{thm:wang}. For $P=P(2\ell,
q,r)$ with $q,r$ odd then we have the following simple consequence of
the above statement:
\begin{cor}\label{cor:ThreeStrand}
  For a three-stand pretzel knot $P$ either the genus $g(P)$ is
  different from 2, or up to mirroring it is $P(2\ell, 3,1)$,
  $P(2\ell, 3,-3)$ or $P(2\ell, -5,1)$ for some $\ell\in \Z$. \qed
  \end{cor}

\subsection{All $a_i$'s are odd}
The following theorem of Gabai describes the genus of an $n$-strand pretzel knot
with all coefficients odd for a general (odd) $n$.
Recall that we always assume
that $\{ a_i\}_{i=1}^n$ cannot contain both $1$ and $-1$.

\begin{thm}(Gabai, \cite[Theorem~3.2]{gabai})
  \label{thm:gabai}
  Suppose that $P=P(a_1, \ldots , a_n)$ is an $n$-strand pretzel knot
  with $n \geq 3$ and all $a_i$ odd, and there are no two indices
  $i,j$ with $a_ia_j=-1$. Then the genus $g(P)$ is equal to
  $\frac{1}{2}(n-1)$. In particular, $g(P)=2$ if and only if
  $n=5$. \qed
\end{thm}

\subsection{The first coefficient $a_1$ is even}
In this case, work of Kim-Lee provides a bound (and often a formula)
for the genus of $P=P(a_1, \ldots , a_n)$ (with $a_1$ even and all
$a_i$ with $i>1$ odd). We will again assume that $\{ a_i\}_{i=1}^n$
does not contain both $1$ and $-1$, $a_1\neq 0$ and if $a_1=\pm 2$
then there is no further $a_i$ which is equal to $\mp 1$.  By
determining the Alexander-Conway polynomial $\nabla _P(z)$ of $P$ and
identifying its leading coefficient, the following bound on the
Seifert genus $g(P)$ has been proved:

\begin{thm}(Kim-Lee, \cite[Theorem~4.1]{genera-pretzel})
\label{thm:gen-pretzel}
  Suppose that the pretzel knot $P=P(a_1, \ldots , a_n)$ has
  $a_1$ even ($\neq 0$), which (by possibly taking the mirror) can be assumed
  to be positive.  Let
  $\alpha =\sum _{i=2}^n {\rm {sign}} (a_i)$ and $\delta =\sum _{i=2}^n
  (\vert a_i\vert -1)$. Then the genus $g(P)$ of $P$ is bounded from below by
  \begin{itemize}
  \item $\frac{1}{2}(\delta +2)$ if $n$ is odd and $\alpha \neq 0$.
    \item $\frac{1}{2}\delta $ if $n$ is odd and $\alpha = 0$. 
\item $\frac{1}{2}(a_1 +\delta )$ if $n$ is even and
  $\alpha \neq -1$.
  \item $\frac{1}{2}( a_1 +\delta )-1$ if $n$ is even and
    $\alpha = -1$.
  \end{itemize}
  In addition, if none of the $a_i$ are equal to $\pm 1$, then the bounds above
  provide the precise value of the genus $g(P)$. \qed
  \end{thm}
  A simple consequence of the above result is:
  \begin{cor}
\label{cor:higherstrands}
    The pretzel knot $P=P(a_1, \ldots , a_n)$ with $a_1\neq 0$ even and
    $a_i$ odd ($i>1$) and with $n\geq 4$
    has genus $>2$ unless
    \begin{enumerate}
\item all $a_i$ with $i>1$ is either $1$ or $-1$ (all these with the
  same sign),
\item $n$ odd, $\alpha \neq 0$, $a_1=2\ell$, $a_2=\pm 3$ and for
  $i>2$ all $a_i=\pm 1$ (all these with the same sign);
    \item $n$ odd, $\alpha =0$, $a_1=2\ell$, $a_2=\pm 3$, $a_3=\pm 3$
    and for $i>3$ all $a_i=\pm 1$ (all with the same sign),
    \item $n$ even, $\alpha \neq -1$, $a_1=2$, $a_2=\pm 3$
      and for $i>2$ all $a_i=\pm 1$ (all with    the same sign),
    \item $n$ even, $\alpha =-1$, $a_1=4$, $a_2=\pm 3$,
      and for $i>2$ all $a_i=\pm 1$ (all with    the same sign),
    \item $n$ even, $\alpha =-1$, $a_1=2$, $a_2=\pm 5$,
      and for $i>2$ all $a_i=\pm 1$ (all with    the same sign).
    \item $n$ even, $\alpha =-1$, $a_1=2$, $a_2=\pm 3$, $a_3=\pm 3$,
      and for $i>3$ all $a_i=\pm 1$ (all with    the same sign). \qed
      \end{enumerate}
    \end{cor}

\section{PCSC for pretzel knots with $n\neq 5$ strands}
\label{sec:generalStrands}
In this section we start proving Theorem~\ref{thm:main1}. First we
deal with those pretzel knots where $n\neq 5$, or when $n=5$ and  the
first coefficient $a_1$ is even.

\subsection{Three-strand pretzel knots}
\label{ss:3strand}
Corollary~\ref{cor:ThreeStrand} gave a list of those three-strand pretzel knots
which have Seifert genus $g(P)=2$.

%Some of those knots have crossing
%number less than 16, hence for those \cite[Theorem~4]{Hanselman} concludes the
%argument. For the knots of type $P(2\ell, q, r)$ we compute the
%Alexander polynomial and use Theorem~\ref{thm:BoyerLines} to conclude
%the argument.

Suppose that the three-strand pretzel knot has one even coefficient
$a_1=2\ell$, which for simplicity is assumed to be negative.  Then by
the repeated application of the skein relation for the Conway
polynomial $\nabla$ we have that (with $\ell <0$)
\[
\nabla _{P(2\ell, q,r )}(z)=\nabla _{P(0, q,r)}(z)+
\vert \ell \vert z \nabla _{T_{2,q+r}}(z).
  \]
      (In the inductive step we used the fact that the 2-component link
      $L_0$ involved in the skein triple is the same torus link
  $T_{2,q+r}$ at every step.)
  Note that  $P(0, q,r)$ is the
  connected sum of two torus knots $T_{2,q}$ and $T_{2,r}$.
  Since $a_2(T_{2,2n+1})={ n+1 \choose 2}$ and
  for the torus link $a_1(T_{2,2m})=\ell k (T_{2,2m})=m$, it follows
  that for $\{ q,r\}=\{ \pm 3 , \pm 1\}, \{ \pm 3, \pm 3\}, \{ \pm 5, \pm 1\}$
  (including all the possible cases of Corollary~\ref{cor:ThreeStrand})
  we get either $a_2(P)\neq 0$ or $\vert \ell \vert $ so small that
  $P(2\ell, q,r)$ is a knot with at most 16 crossing. Since for those
  the PCSC has been verified, we have
  \begin{prop}
    \label{prop:3strandFinal}
    If $P=P(p,q,r)$ is a three-strand pretzel knot, then the purely
    cosmetic surgery conjecture holds for $P$. \qed
    \end{prop}

    \subsection{More than three strands}
We start with the case when $a_1$ is even (and nonzero).

    \begin{thm}\label{thm:many}
      Suppose that $P=P(a_1, \ldots , a_n)$ is an $n$-strand pretzel
      knot with $n\geq 4$ and $a_1$ even, while all $a_i$ with $i>1$
      are odd. Then $P$ satisfies PCSC.
    \end{thm}
    \begin{proof}
      Most of these knots have genus more than 2, hence
      Proposition~\ref{prop:thickness} provides the result.  The
      exceptions (i.e. those pretzel knots considered by the theorem
      which have genus at most 2) are listed in
      Corollary~\ref{cor:higherstrands}, and they can be handled by
      similar means as we did in the case of three-strand knots:
      either they have low crossing number, or the second coefficient
      of the Conway polynomial provides the desired obstruction.

    Indeed, if we have Case (1) of Corollary~\ref{cor:higherstrands},
    then $P$ is a two-bridge knot, and PCSC follows from
    \cite{twobridge}.

    For $n$ odd (cases (2) and (3) in Corollary~\ref{cor:higherstrands})
    the computation of the Conway
    polynomial proceeds exactly as for the three-strand case, providing
    that
\[
\nabla _{P(2\ell, a_2, \ldots , a_n)}(z)=\prod _{i=2}^n \nabla _{T_{2,a_i}}(z)
+\vert \ell \vert z \nabla _{P(a_2, \ldots , a_n)}(z).
\]
By multiplicativity of $\nabla$ under connected sum, we have that
$a_2(\# _{i=2}^n T_{2,a_i})=\sum _{i=2}^n a_2(T_{2,a_i})$ and
$a_2(T_{2,a_i})={\frac{\vert a_i\vert +1}{2} \choose 2}$. Furthermore,
for the two-component link $Q=P(a_2, \ldots , a_n)$ we have
$a_1(Q)=\ell k (Q)=-\frac{1}{2}\sum _{i=2}^n a_i$, where this latter
term is the linking number of the two components of $Q$ (both
unknots). In the cases (2) and (3) the $a_2$-invariants of the torus
knots are 1 (for $T_{2,3}$) and 0 (for the trivial knot), hence the
same argument as for the three-strand case shows that either
$a_2(P)\neq 0$, or the knot has crossing number at most 16, concluding
the argument.

A similar argument works when $n$ is even. Indeed, we can relate
$\nabla _{P(2\ell, a_2, \ldots , a_n)}(z)$ to
$\nabla _{P(0, a_2, \ldots , a_n)}(z)$ by the repeated application of the skein
rule, although this case is slightly different.
Because of the change of the orientation
pattern on the strand with even coefficient, the
link in the skein triple will be different in every step: in
the $i^{th}$ step it will be
$P(2\ell -(2i-1), a_2, \ldots , a_n)$. The expression for
$a_2(P(2\ell, a_2, \ldots , a_n))$ (just as before) will involve a
term $a_2(P(0, a_2, \ldots , a_n))$, which (as before) is the sum
of $a_2$-invariants of alternating torus knots --- mostly the unknot.
The other term now is a sum of the form
$\sum _{i=1}^{\ell } a_1(P(2\ell -(2i-1), a_2, \ldots , a_n))$,
and here the terms are equal to the linking numbers of components of
the two-component links.
In the cases listed under (4)-(7)
in Corollary~\ref{cor:higherstrands} the same scheme will be visible:
there will be only few cases when $a_2$ is zero, and those correspond
to knots with low crossing number, hence the argument is complete. 
\end{proof}

 We close this section with the case when all $a_i$ are odd and $n\geq 6$.   
     \begin{prop}
   \label{prop:nBigOdd}
   If $n\geq 6$ odd and all $a_i$ are odd, then the pretzel knot
   $P(a_1, \ldots , a_n)$ satisfies PCSC.
    \end{prop}
    \begin{proof}
      In these cases Theorem~\ref{thm:gabai} implies that the genus of
      the knot is $\frac{1}{2}(n-1)>2$, hence
      Proposition~\ref{prop:thickness} concludes the argument.
    \end{proof}

\section{Five-strand pretzel knots}
\label{sec:FiveStrand}
Suppose now that $P=P(a_1, \ldots , a_5)$ is a five-strand
pretzel knot with all $a_i$ odd.
Depending on the signs of the coefficients, we will distinguish
two cases.

\subsection{Among the $a_i$'s there are 0,1,4 or 5 negative coefficients}
%First we assume that there are 0,1,4 or 5 negative coefficients
%among $\{ a_i\}_{i=1}^5$.

\begin{lem}
  \label{lem:0145}
  Suppose that the five-strand pretzel knot $P=P(a_1, \ldots , a_5)$ has
  only odd coefficients and among them 0,1,4 or 5
  are negative. Then $\tau (P)\neq 0$.
\end{lem}
\begin{proof}
 As the proof of Proposition~\ref{prop:thickness} shows, in
 these cases the two possible $\delta$-gradings are 3 and 2 (if there
 are only positive coefficients),  2 and 1 (if there is a unique
 negative coefficient), and symmetrically
 $-2$ and $-1$ in case of a unique positive coefficient, and
 $-3$ and $-2$ when there are five negative coefficients.
 Recall that $\tau (P)$ is the Alexander
 grading of one of the homogeneous elements of $\HFKa (P)$ with Maslov
 grading 0. In case $\tau (P)=0$, there should be an element
 with $\delta$-grading 0, a contradiction. Therefore in these cases
 $\tau (P)\neq 0$.
\end{proof}

\begin{prop}
  \label{prop:0145}
  Suppose that the five-strand pretzel knot $P=P(a_1, \ldots , a_5)$ has
  only odd coefficients and among the five odd coefficients  0,1,4 or 5
  are negative. Then $P$ satisfies PCSC.
  \end{prop}
\begin{proof}
  Since in these cases by Proposition~\ref{lem:0145} we have that
  $\tau (P)\neq 0$,  Theorem~\ref{thm:niwu} implies the result.
  \end{proof}

\subsection{There are 2 or 3 negative coefficients among the $a_i$'s}
In this case our arguments will rest on the obstructions
stemming from the coefficient $a_2$ of the Conway polynomial, together
with the $w_3$-invariant introduced in Section~\ref{sec:Obstructions}.
Since the coefficients of
$P=P(a_1, \ldots , a_5)$ are all odd, there is an
obvious Seifert surface of genus two associated to the diagram of the
knot given in Figure~\ref{fig:Pretzel}.
%(Indeed, by Gabai's result Theorem~\ref{thm:gabai}, except in the
%cases we exlude by assuming
%$\{ 1, -1\}$ is not a subset of $\{ a_i\}_{i=1}^5$, this Seifert surface is
%of minimal genus.)
The Seifert matrix in the obvious basis is given
in \cite[Section~2.1]{Varva}, where it has been also shown 
that
\begin{prop}(\cite[Lemma~2.2]{Varva})
\label{prop:a2Formula}
  Suppose that $P=P(a_1, \ldots , a_5)$ is a five-strand pretzel
  knot with $a_i=2k_i+1$ odd. Then
  \[
  a_2(P)=s_2+2s_1+3,
  \]
  where $s_i$ is the value of the $i^{th}$ elementary symmetric
  polynomial in five variables evaluated on $\{ k_1, \ldots ,
  k_5\}$. \qed
\end{prop}

    Using the skein rule, a formula for $v_3(K)=-2w_3(K)$ has been given
    in \cite[Lemma~2.2]{Varva} for all pretzel knots with odd coefficients.
    For a five-strand pretzel  knot $P=P(2k_1+1, \ldots , 2k_5+1)$
    the result provides

    \begin{lem}(\cite[Lemma~2.2]{Varva})
      \label{lem:v3}
      $w_3(K)=\frac{1}{2}(5+3s_1+s_1^2+s_2+\frac{1}{2}(s_3+s_1s_2))$,
      where the values of the elementary symmetric polynomials $s_1,
      s_2,s_3$ are as given in Proposition~\ref{prop:a2Formula}. \qed
      \end{lem}

    \begin{rem}
      The statements of Proposition~\ref{prop:a2Formula} and
      Lemma~\ref{lem:v3} in \cite{Varva} have been formulated for the
      case of $k_i\geq 0$; the proofs of these statements, however, hold in
      the wider generality we use them here.
      \end{rem}
    
    With these preparations in place, we can now turn to the
    verification of PCSC for five-strand pretzel knots.
    
    \begin{prop}
      \label{prop:fivestrandcase}
      Suppose that $P=P(a_1, a_2, a_3, a_4, a_5)$ is a five-strand
      pretzel knot with all coefficients odd. Then the purely cosmetic
      surgery conjecture holds for $P$.
      \end{prop}
    
    \begin{proof}
      We can assume that there are two or three negative coefficients
      among the $\{ a_i\}_{i=1}^5$, since (by
      Proposition~\ref{prop:0145}) in the other cases PCSC holds.  If
      $P$ has $a_2(P)\neq 0$, then Theorem~\ref{thm:BoyerLines}
      implies the result. If $a_2(P)=0$ and $w_3(P)\neq 0$, then
      Theorem~\ref{thm:w3application} concludes the argument.  Suppose
      therefore that $P=P(2k_1+1, \ldots , 2k_5+1)$ has $a_2(P)=0$
      (implying that $s_2=-2s_1-3$) and $w_3(P)=0$, implying in the
      light of Lemma~\ref{lem:v3} (after substituting $s_2=-2s_1-3$)
      that $s_3=s_1+2$.

   By using the standard identities
      \[
      \sum _{i=1}^5 k_i^2 =s_1^2-2s_2, \qquad \sum _{i=1}^5 k_i^3=s_1^3-3s_1s_2+3s_3,
      \]
      and substituting $s_2=-2s_1-3$ and  $s_3=s_1+2$, we get
\[
\sum_{i=1}^5 k_i^2 = s_1^2+4s_1+6=(s_1+2)^2+2, \qquad
\sum _{i=1}^5k_i^3 = s_1^3+6s_1^2+12s_1+6=(s_1+2)^3-2.
\]
Let $S=\sum _{i=1}^5a_i$. Since $S=2s_1+5$, we get that
\[
\sum a_i^2=S^2+4, \qquad \sum a_i^3=S^3.
\]

Let 
\[
\posit =\{ i\in \{ 1,\ldots ,5\}\mid a_i>0\}, \qquad
\negat =\{ i\in \{ 1,\ldots ,5\}\mid a_i<0\}.
\]

By our assumption on the signs of the $k_i$, we can assume that
both $\posit$ and $\negat$ have two or three elements, implying
that
\begin{equation}\label{eq:posneg}
\sum _{i\in \posit}a_i^3>0, \qquad \sum _{i\in \negat}a_i^3<0.
\end{equation}

We can also assume that one of the two inequalities
\begin{equation}\label{eq:bounds}
\sum _{i\in \posit }a_i^2\leq S^2, \qquad \sum _{i\in \negat }a_i^2\leq S^2
\end{equation}
holds, since the violation of both would imply $2S^2\leq S^2+4$, hence
$S^2\leq 4$, so $\sum _{i=1}^5a_i^2\leq 8$, therefore $P$ is a knot of
crossing number less than 16, for which PCSC holds true.

Assume first that \emph{both} inequalities in Equation~\eqref{eq:bounds}
are satisfied. In this case $\vert a_i \vert \leq \vert S\vert$,
hence
\[
\sum _{i\in \posit }a_i^3\leq \sum _{i\in \posit }\vert S\vert a_i^2\leq \vert S\vert ^3
\]
and
\[
\sum _{i\in \negat }a_i^3\geq \sum _{i\in \negat }-\vert S\vert a_i^2\geq -\vert S\vert ^3.
\]

Combining these inequalities with the ones from Equation~\eqref{eq:posneg}
we get
\[
-\vert S\vert ^3<\sum _{i=1}^5 a_i^3< \vert S \vert ^3,
\]
providing a contradiction to $\sum _{i=1}^5 a_i^3=S^3$.  This
shows, that under the assumptions that both 
inequalities of Equation~\eqref{eq:bounds} hold, if
$a_2(P)=0$ then $w_3(P)\neq 0$.

Assume now that one of the inequalities of Equation~\eqref{eq:bounds}
is false. This implies that terms in the other inequality sum up to at
most 3, implying that all terms in this other inequality satisfy
$a_i^2=1$, i.e. $a_i=\pm 1$ (with the same sign).  By possibly
mirroring the knot, we can assume that these terms are all equal to 1,
hence the corresponding $k_i=0$. By our previous assumption, there
are two or three such coefficients.

{\bf {Case I}}: Suppose first that there are three positive
coefficients $a_1=a_2=a_3=1$ and $a_4,a_5<0$.  This implies that
$k_1=k_2=k_3=0$, hence when computing $a_2(P)$, we get that it is
equal to $3+2(k_4+k_5)+k_4k_5$, while the expression $s_3-s_1-2$ is
equal to $-k_4-k_5-2$. If the corresponding pretzel knot violates
PCSC, both expressions need to be zero, and we get $k_4k_5=1$, hence
$k_4=k_5=-1$. Since $(a_1, \ldots , a_5)=(1,1,1,-1,-1)$ gives the
unknot, we can ignore this case.

{\bf {Case II}}: Suppose that there are two positive coefficients $a_1=a_2=1$,
and $a_3,a_4,a_5<0$. With the usual definition of $k_i$
as $a_i=2k_i+1$, we have that $k_1=k_2=0$ and $a_2(P)=k_3k_4+k_3k_5+k_4k_5+
2(k_3+k_4+k_5)+3$ and $s_3-s_1-2=k_3k_4k_5-k_3-k_4-k_5-2$. If one of them
is nonzero, then $P$ satisfies PCSC. If both are zero, then so is their sum:
\[
2k_3k_4k_5+k_3k_4+k_3k_5+k_4k_5-1=0.
\]
Writing this sum as
\begin{equation}\label{eq:vegsosum}
k_3k_4(k_5+1)+k_3k_5(k_4+1)+k_4k_5-1,
\end{equation}
the first two terms are negative unless $k_5=-1$ or $k_4=-1$, in which
cases the knot has (at most) three strands; the same applies if
$k_3=-1$. Since $k_3(k_5+1) >\vert k_5\vert $ or $k_3(k_4+1) > \vert
k_4\vert $ once $k_3<-1$, the expression of Equation~\ref{eq:vegsosum}
is negative, providing the desired contradiction.
\end{proof}

%\begin{rem}
%  There are partial cases which can be delt with simple divisibility
%  properties of the $k_i$'s. Write all odd $a_i$ as $a_i=2k_i+1$ as
%  before, and assume that 0,1,4 or 5 of the integers $k_i$ are
%  odd. Then $s_2$ (as defined above) is even, hence $a_2(P)=3+2s_1+s_2$ is odd
%  (and therefore is nonzero). In the light of
%  Theorem~\ref{thm:BoyerLines} this also implies Theorem~\ref{thm:main1}. A
%  similar argument (using mod 4 arithmetic) concludes the argument if
%  two of the $k_i$'s are odd, but they differ mod 4.
%\end{rem}    

\begin{proof}[Proof of Theorem~\ref{thm:main1}]
    The proof of the theorem for the case of $n=3$ is provided by
    Proposition~\ref{prop:3strandFinal}.  When $a_1$ is even and
    $n\geq 4$, the result is proved in Theorem~\ref{thm:many}.  When
    $n\geq 6$ and all $a_i$ are odd, Proposition~\ref{prop:nBigOdd}
    gives the result.  Finally in the cases when $n=5$ and all $a_i$ odd,
    Proposition~\ref{prop:fivestrandcase} verifies the claim. This
    completes the proof of Theorem~\ref{thm:main1}.
\end{proof}

\section{Appendix: the Jones polynomial for pretzel knots}
In this section we provide a convenient formula for the Jones
polynomial of pretzel knots with odd coefficients.
Recall that the Jones polynomial $V_K(t)$ is defined by the skein relation
\[
t^{-1}V_{L_+}(t)-tV_{L_-}(t)=(t^{\frac{1}{2}}-t^{-\frac{1}{2}})V_{L_0}(t)
\]
and normalization $V_{U}(t)=1$ on the unknot $U$.

Suppose that $P=P(a_1,\ldots , a_n)$ is an $n$-strand pretzel knot
with $a_i$ odd. Let $s=t^{\frac{1}{2}}$, $k\in \Z$ be an integer and
$v_i\in \{ 0,1\}$.
We define functions $P_{v_i,k}(s)$ as
follows. For $v_i=0$ take
\[
P_{0,k}(s) = -s^{-2k}.
\]
If $v_i = 1$ and $k >0$, take
\[
P_{1,k}(s) = \sum_{j=1} ^k (-1)^j\cdot s^{1-2j};
\]
and if $v_i = 1$ and  $k <0$, take
\[
P_{1,k}(s)= \sum _{j = 1} ^{-k} (-1)^j\cdot s^{-1+2j}.
\]
For a fixed vector $v\in \{ 0,1\}^n$ multiply the terms $P_{v_i,a_i}(s)$
corresponding
to the twisting numbers $a_1, \ldots , a_n$ of the given pretzel knot,
and multiply the result with the Jones polynomial of the
$d(v)$-component unlink, where $d(v) = |(n-1)- \sum_{i= 1} ^n v_i |$,
resulting in
\[
Q_{v,a_1, \ldots , a_n} (s) = (-s-s^{-1})^{d(v)}\cdot P_{v_1,a_1}(s)\cdot P_{v_2,a_2}(s)
\cdots P_{v_n,a_n}(s).
\]
Finally, add these terms and get 
$W_{P}(s) =\sum _{v\in \{ 0,1\}^n} Q_{v,a_1, \ldots , a_n}(s)$.
The verification of the fact that we get the Jones polynomial follows the
same route as the description of the Jones polynomial
through spanning tree expansion, as given in \cite{Morwen}.

\begin{prop}
  With the substitution $t = s^2$ the function $W_P(s)$ provides the
  Jones polynomial $V_P(t)$ of the $n$-strand pretzel knot $P$ with all
  odd coefficients. \qed
  \end{prop}

\begin{rem}
  This formula can also be used to prove the formula of Lemma~\ref{lem:v3}.
  \end{rem}

%Notice that the same formulae work regardless of the parity of $n$,
%providing a formula for the Jones polynomial of the two-component
%pretzel link $P(a_1, \ldots , a_n)$ with all $a_i$ odd and $n$ even.

%Using the usual skein relation on the first strand $\vert \ell
%\vert$-times, together with the above calculations, we get a formula
%for the Jones polynomial of the pretzel knot $P(2\ell, a_2, \ldots ,
%a_n)$ (with all $a_i$ with $i>1$ odd).

\bibliography{biblio} \bibliographystyle{plain}

\end{document}